\documentclass[twoside, 10pt]{article}
 \RequirePackage{amsgen}
\RequirePackage{amsmath} \RequirePackage{amstext}
\RequirePackage{amsbsy} \RequirePackage{amsopn}
\RequirePackage{amsthm} \RequirePackage{amssymb}
\RequirePackage{epsfig}

\usepackage[mathscr]{eucal}

\theoremstyle{plain}
\newtheorem{theorem}{Theorem}[section]
\newtheorem{proposition}[theorem]{Proposition}
\newtheorem{lemma}[theorem]{Lemma}
\newtheorem{definition}[theorem]{Definition}
\newtheorem{corollary}[theorem]{Corollary}

\newtheorem{remark}[theorem]{Remark}

\let\a\alpha
\let\b\beta

\let\e\epsilon
\let\d\delta

\let\ph\phi
\let\e\varepsilon
\let\l\lambda

\let\k\kappa

\let\ph\phi

\def\V{\mathcal V}
\def\O{\mathcal O}

\def\CF{\mathcal F}

\def\D{\Lambda}
\def\sD{\mathcal D}

\def\F{{\mathbb F}}

\def\Fq{{\F_{q}}}
\def\Z{{\mathbb Z}}
\def\Zp{{\Z_{p}}}
\def\Q{{\mathbb Q}}
\def\Qp{{\Q_p}}

\def\O{\mathcal O}

\def\Z{{\mathbf Z}}
\def\Zp{{\mathbf Z}_p}
\def\F{{\mathbf F}}
\def\Fq{{\mathbf F}_q}

\def\Z{{\bf Z}}
\def\Q{{\bf Q}}
\def\p{\pi}
\def\D{\Delta}

\def\a{\alpha}
\def\l{\lambda}

\def\b{\beta}

\def\cD{\mathcal D}

\def\S{\mathbf S}

\begin{document}

\title{{\bf Cartier operators  \\ 
on fields of  positive characteristic $p$  }}
\author{{ Sangtae Jeong}
\\Department of Mathematics, Inha University, Incheon, Korea
402-751}

\footnotetext{


\noindent {\rm Email}:  stj@inha.ac.kr }

\bibliographystyle{alpha}

\maketitle

\begin{abstract}
From an analytical perspective, we introduce a sequence of Cartier operators that act on the field of formal Laurent series
in one variable with coefficients in a field of positive characteristic $p.$ In this work, we discover the binomial inversion formula between Hasse derivatives and Cartier operators, implying that Cartier operators can play a prominent role in  various objects of study in function field arithmetic, as suitable substitutes for higher derivatives. For an applicable object, the Wronskian criteria associated with Cartier operators are introduced.
These results stem from a careful study of two types of Cartier operators on the power series ring $\Fq[[T]]$ in one variable $T$ over a finite field $\Fq$ of $q$ elements.  Accordingly, we show that two sequences of Cartier operators are an orthonormal basis of the space of continuous $\Fq$-linear functions on $\Fq[[T]].$ According to the digit principle, every continuous function on $\Fq[[T]]$ is uniquely written in terms of a $q$-adic extension of Cartier operators, with a
closed-form of expansion coefficients for each of the two cases. Moreover, the $p$-adic analogues of Cartier operators are discussed as orthonormal bases for the space of continuous functions on $\Zp.$
\end{abstract}

\tableofcontents

\section{\bf Introduction}
A Cartier operator is of great importance in characteristic-$p$ algebraic geometry, which is a fundamental tool for working with K$\ddot{a}$hler differential forms in this geometry. It also plays a significant role in determining the criterion of algebraicity of power series over the field of rational functions in a field of characteristic $p>0$ as is shown in studies conducted by Christol \cite{Ch} and Sharif and Woodcock \cite{SW}. Current research pertaining to Cartier operators is of an algebraic and arithmetic nature. In this paper, from a purely analytical perspective, we consider two types of Cartier operators (or maps) which act on a non-Archimedean local field of any characteristic, while much emphasis is placed on the power series ring $\Fq[[T]]$ in one variable $T$ over a finite field $\Fq.$

The purpose of this work is to apply two types of Cartier maps to the characterization of continuous functions defined on the integer ring  of a non-Archimedean local field of any characteristic. Accordingly, we first proceed to deduce a fundamental relation of what is known as a binomial inversion formula between Cartier operators and higher derivatives (or Hasse derivatives) on $\Fq[[T]].$ This binomial inversion formula holds over a more general field of formal Laurent series in one variable over a (perfect) field of characteristic $p>0$, which shows that Cartier operators can play a role in various objects of study as a substitute for higher derivatives. For example, in Section 5, we present selected Wronskian criteria associated with two Cartier operators on more general fields, and these are parallel to the Wronskian criteria for the Hasse derivatives  in \cite{Sc} and \cite{GV}.
As compared with the known properties of Hasse derivatives whose $q$-adic extension is referred to as digit derivatives \cite{J1,J3,J4}, we show in Sections 2 and 3 that two sequences of Cartier operators are an orthonormal basis for the closed subspace $LC({\F}_q[[T]],{\F}_q((T)))$ of $\Fq$-linear continuous functions on ${\F}_q[[T]].$
Added to this fact, Conrad's digit principle \cite{Co2} enables us to prove that $q$-adic extensions of two Cartier operators are an orthonormal basis for the entire space, $C({\F}_q[[T]],{\F}_q((T)))$, of continuous functions on ${\F}_q[[T]]$ together with a closed-form expression for expansion coefficients for two respective bases. At the same time, by analogy with the classical case, we provide two orthonormal bases for the space $C(\Zp, \Qp)$ of continuous functions on the ring $\Z_p$ of $p$-adic integers, consisting of $p$-adic extensions of two types of Cartier maps on $\Zp$ with no closed-form formula for expansion coefficients.

\section{Orthonormal bases for $LC(R,K).$}
This section consists of three subsections. The first subsection is a quick review of orthonormal bases of
a certain Banach space over the integer ring of a local field of any characteristic, with  much emphasis on the positive
characteristic. For such a Banach space, we mainly consider $LC(R,K)$ (see notational exposition after Lemma \ref{ba}).
Two types of Cartier operators are the main object of study in the second subsection, forming an orthonormal basis
of the space $LC(R,K)$. In the last subsection, we observe that all known orthonormal bases of $LC(R,K)$ are essentially
equivalent.

\subsection{Known bases for the subspace on $\Fq[[T]]$}

Let $\V$ be a non-Archimedean local field of any characteristic, with an integer ring $\O$ and  a maximal ideal $M.$
In addition, let $\p$ be a uniformizer in $\V$  such that $M=(\p),$ and let $\F:=\O/M$ be
the residue field of order $q,$ and let $|?|_{\pi}$  be the (normalized) absolute value on $\V$ associated with the additive valuation ${v}_{\p}$ on $\V$ such that $ |x|_{\p} = q^{-{v}_{\p}(x)}$ for $x \in \V.$

The following are the cases for $(\V,\O,\p)$ which are of greatest interest in this study:

\indent (1) $ (\F_q((T)),\F_q[[T]], T),$ where $\F_q((T))$ is the field of formal Laurent series in one variable $T$ over a finite field $\Fq$ of $q$ elements, where $q=p^e$ is a power of a prime number $p$ and $\F_q[[T]]$ is the ring of the formal power series in $T$ over  $\Fq.$

\indent (2) $(\Q_p, \Z_p, p),$ where $\Q_p$ is the field of  $p$-adic numbers for a prime number $p$ and $\Zp$ is the ring of  $p$-adic integers.

\begin{definition}
Let $K$ be a non-Archimedean local field, and let $E$ be a $K$-Banach space equipped
with the usual sup-norm. We say that a sequence $\{f_n\}_{n\geq0}$
in $E$ is an orthonormal basis for $E$ if and only if the
following two conditions are satisfied:
\par
{\rm(1)} every $f \in E$ can be expanded uniquely as $f
=\sum_{n\geq0}a_{n}f_n$, with $a_{n} \in K \rightarrow 0$ as $n \rightarrow\infty.$
\par
{\rm(2)} The sup-norm of $f$ is given by $\|f\|
=\mbox{max}\{|a_{n}|\}.$
\end{definition}
For subsequent use we state a simple, yet useful criterion of an orthonormal basis for a $K$-Banach space $E$ which
follows immediately from Serre's criterion \cite[Lemme I]{Se}.
\begin{lemma}\label{ba}
Let $K$ be a non-Archimedean local field with a nontrivial
absolute value and let $E$ be a $K$-Banach space with an orthonormal
basis
 $\{e_{n}\}_{n\geq0}.$ If $f_n \in E$ with
 $\sup_{n\geq0} \| e_{n}-f_{n}\| <1 $
 then $\{f_n \}_{n\geq0}$ is an orthonormal basis of $E.$
\end{lemma}

\begin{proof}
See \cite[Lemma 3.2]{Co1} for an alternative proof.
\end{proof}

In what follows, let $R=\Fq[[T]]$ and $K=\Fq((T))$ and let
$C(R,K)$ denote the $K$-Banach space of all continuous functions $f$ : $R \rightarrow
{K}$ equipped with the sup-norm $\|f\| = \displaystyle \max_{x
\in R}\{|f(x)|\}.$ Unlike the classical case, $C(R,K)$ contains a subspace of continuous
${\F}_q$-linear functions from $R$ to $K$, which is denoted by $LC(R,K).$

We now provide a brief review of three sets of well-known  orthonormal bases for $E$ in the case where
$E=LC(R,K).$ First, the Hasse derivative $\{\cD_{n}\}_ {n\geq 0}$ on $R$ is a sequence of functions
defined by $$\cD_{n}(\sum_{i\geq0} x_{i}T^i) = \sum_{i\geq n} \binom{i}{n}x_{i}T^{i-n}.$$
As is shown in  \cite{V}, $\{\cD_{n}\}_{n\geq0}$ is a continuous
${\F}_q$-linear operator on $R$ and satisfies various properties for higher differentiation rules.
To recover the expansion coefficients let us recall the Carlitz difference operators $\{{\D}^{(n)}\}_{n\geq 0}$ on $LC(R,K),$ which are defined
recursively by
\begin{eqnarray*}\label{do}
({\D}^{(n)}f)(x)= {\D}^{(n-1)}f(Tx)-T^{q^{n-1}}{\D}^{(n-1)}f(x)
 (n \geq 1); {\D}^{0} =id.
\end{eqnarray*}
For simplicity, $\D$ and $\D^{n}$ denote $\D^{(1)}$ and the $n$th iterate of $\D$, respectively.

\begin{theorem}\label{Hre}
(1) $\{\cD_{n}\}_ {n\geq 0}$ is an orthonormal basis for  $LC(R,K)$.

(2) Write $f =\sum_{n=0}^{\infty}b_{n}\cD_{n}\in LC(R,K).$ Then, the coefficients can be
recovered by iterating the Carlitz difference operator ${\D}$:
\[
b_{n} =({\D}^{n}f)(1) =\sum_{i=0}^{n}(-1)^{n-i}f(T^i)\cD_{i}(T^n).
\]
\end{theorem}
\begin{proof}
See \cite{J1, J3} or \cite{Sn}.
\end{proof}

Secondly, we define the Carlitz $\Fq$-linear polynomial $\{E_n\}_{n\geq0},$
which is given by
$$E_n(x)=e_n(x)/F_n~(n\geq1)~~{\rm and}~~ E_0(x) =x,$$
where  for $n\geq1,$ $$e_{n}(x) =
\prod_{\scriptsize{\begin{array}{l}
                 \a \in {\F}_{q}[T] \nonumber \\
                 \mbox{deg}(\a) < n \nonumber
                   \end{array}}}(x-\a) $$
and
 $$F_{n} =[n]{[n-1]}^q \cdots {[1]}^{q^{n-1}};~ L_{n}
=[n][n-1]\cdots [1] (n>0);~F_0=L_0=1,$$
where $[n] =T^{q^n} -T(n>0).$


\begin{theorem}\label{Wre}
(1) $\{E_n(x)\}_{n\geq0}$ is an orthonormal basis for  $LC({R},K)$.

(2) Write $f =\sum_{n\geq0}a_{n}E_{n}(x)\in LC(R,K)$. Then the coefficients can be
recovered by the formula:
\begin{eqnarray}\label{carcoeff}
a_{n} =({\D}^{(n)}f)(1)=\sum_{i=0}^n C_i f(T^i),
\end{eqnarray}
where \begin{eqnarray}\label{valCn}
C_n=1;~~~~ C_{i}=(-1)^{n-i}\sum_{e \in S_i}T^e ~~(0\leq i< n),
 \end{eqnarray}
where $S_i$ is the set of all sums of distinct elements of $\{1, q \cdots, q^{n-1}\}$ taken $n-i$ at a time.
\end{theorem}
\begin{proof}
See \cite{W2,W1} or \cite{Co2} and \cite{J1, J3}.
\end{proof}
We point out here that the coefficients in (\ref{carcoeff}) are also recovered by the formula
\begin{eqnarray}\label{carcof2}
a_{n} =\sum_{i=0}^{n}\sum_{r=i}^{n}\binom{r}{i}(-T)^{r-i}A_{n,r}f(T^i),
\end{eqnarray}
where $A_{n,1} =(-1)^{n-1}L_{n-1}$ and for $ r>1$,
\begin{eqnarray*}\label{Anr}
A_{n,r} &=& (-1)^{n+r}L_{n-1} \sum_{0 < j_1 \cdots <j_{r-1}
<n}\frac{1}{[j_1] [j_2]\cdots [j_{r-1}]},
\end{eqnarray*}
with the convention that
$A_{0,0}=1$ and $A_{n,0}=1$ for $n>0.$
Indeed the formula in (\ref{carcof2}) follows from the identity ${\D}^{(n)} =\sum_{r=0}^n A_{n,r}\D^r$
in \cite[Proposition 3]{J2} and the formula (2) in Theorem \ref{Hre}.
By comparing two formulas for $a_n$ in (\ref{carcoeff}) and (\ref{carcof2})
the combinatorial sum in (\ref{valCn}), for all $0 \leq i \leq n,$ is given by the formula
$$ \sum_{e \in S_i}T^e = \sum_{r=i}^{n}(-1)^{n-r}\binom{r}{i}T^{r-i}A_{n,r}.$$


Finally, as an $\Fq$-linear operator on $R,$ the shift map $\{ \S^{(n)}\}_{n\geq 0}$ is defined by
\begin{eqnarray*}\label{sn}
\S^{(n)}(\sum_{i \geq0} x_{i}T^i ) &=& \sum_{i\geq n} x_{i}T^{i-n}.
\end{eqnarray*}

\begin{theorem}\label{snrep}
(1) $\{ \S^{(n)} \}_ {n\geq 0}$
is an orthonormal basis for  $LC(R,K).$

(2) Write $f=\sum_{n \geq0}c_n \S^{(n)}(x)\in LC(R,K).$ Then the coefficients $c_n$ are given by
the formula:
\begin{eqnarray*}\label{cnshift}
c_0 &= & f(1); \nonumber \\
c_n&= & f(T^n) -Tf(T^{n-1})~(n\geq1).
\end{eqnarray*}
\end{theorem}
\begin{proof}
See \cite{J6}.
\end{proof}

\subsection{Cartier operators on $\Fq[[T]]$}
We begin by introducing the Cartier operator on $\Fq[[T]]$ that is the main object of this study.
\begin{definition}
For $m$ and $r$ integers such that  $0 \leq r <q^m,$ the Cartier operator $\D_{r, m}$ on $\Fq[[T]]$ is defined by
 $$ \D_{r,m}(\sum_{n\geq0} x_nT^n) = \sum_{n\geq0}x_{nq^m+r}T^n.$$
\end{definition}
Observe that $\D_{r, m}$ is a complete generalization of $\D_{r,1},$ which is defined in \cite{AS}.
Note also that $\D_{r,m}$ is  an $\Fq$-linear operator on  $\Fq[[T]]$ and that
it can be defined, for a monomial $T^n(n\geq 0),$ by
\begin{equation}\label{directdef}
\D_{r,m}(T^n) = \left \{ \begin{array}{ll}
    T^l & \mbox{if}~~ n=lq^m+r
\\ 0   & \mbox{otherwise}
 \end{array}
 \right.
\end{equation}
and then it can be extended to $\Fq[[T]]$ by $\Fq$-linearity.


The relevant basic properties of the operators $\D_{r,m}$ are stated as follows.

\begin{lemma}\label{basic}
For  $x , y \in \Fq[[T]],$
\begin{itemize}
\item[] {\rm (1)} $ x =\sum_{r=0}^{q^m-1}T^r\D_{r,m}^{q^m}(x).$
\item[] {\rm (2)} $ \D_{r,m}(x^{q^m} y)=x\D_{r,m}(y).$
\item[] {\rm (3)} For all integers $s\geq1$ such that $r+s<q^m,$
    $$\D_{r,m}(x)=\D_{r+s,m}(T^sx).$$
\end{itemize}
\end{lemma}

\begin{proof}
The proofs of parts (1) and (2) follow from those of the case $\D_{r,1}$ in \cite[Lemma 12.2.2]{AS}. For completeness, we include a proof here.
For (1), we have
\begin{eqnarray*}
 x &=& \sum_{n\geq 0}x_nT^n =\sum_{r=0}^{q^m -1}\sum_{n\geq0}x_{nq^m +r}T^{nq^m +r}\\
 &=& \sum_{r=0}^{q^m -1}T^r\left(  \sum_{n\geq0}x_{nq^m +r}T^{n} \right)^{q^m}= \sum_{r=0}^{q^m -1}T^r\D_{r,m}^{q^m}(x).
\end{eqnarray*}
For (2), we have
\begin{eqnarray*}
x^{q^m}y &=& (\sum_{i\geq0}x_iT^i)^{q^m} \sum_{j\geq0} y_jT^j=\sum_{i\geq0}x_iT^{iq^m} \sum_{j\geq0}^{\infty} y_jT^j= \sum_{n\geq0} T^{n}\left(\sum_{\substack{i,j\geq0 \\iq^m +j=n}}x_{i}y_j\right).
\end{eqnarray*}
Hence,
\begin{eqnarray*}
\D_{r,m}(x^{q^m} y) &=&\sum_{n\geq0}T^{n}\left( \sum_{\substack{i,j\geq0 \\iq^m +j=nq^m +r}}x_iy_{j}\right)
=\sum_{n\geq0} T^{n}\left( \sum_{0\leq i \leq n}x_iy_{(n-i)q^m +r}\right) \\
&=&\left( \sum_{i\geq0} x_iT^{i}\right) \left( \sum_{n \geq i}y_{(n-i)q^m +r}T^{n-i}\right) = x\D_{r,m}(y).
\end{eqnarray*}

Part (3) follows from the repeated application of the identity $\D_{r,m}(x) =\D_{r+1,m}(Tx),$ which is obtained
from (\ref{directdef}).
\end{proof}

The following result shows that $\D_{r,m}$ satisfies the product formula.
\begin{lemma}\label{prodf}
For  $x , y \in \Fq[[T]],$
$$\D_{r,m}(xy) = \sum_{i+j=r}\D_{i,m}(x) \D_{j,m}(y) + T\sum_{i+j=q^m+ r}\D_{i,m}(x) \D_{j,m}(y).$$
In particular,
$$\D_{q^m -1,m}(xy) = \sum_{i+j=q^m -1}\D_{i,m}(x) \D_{j,m}(y).$$
\end{lemma}
\begin{proof}
This follows from the identity in Lemma \ref{basic} (1) which also provides the uniqueness of such a representation
of any element in $\Fq[[T]],$
in terms of $\D_{r,m}^{q^m}.$ This is left to the reader to verify.
\end{proof}

\begin{lemma}\label{continu}
For $0 \leq r <q^m,$ $\D_{r,m}$ is continuous on  $\Fq[[T]].$
\end{lemma}

\begin{proof}
Because $\D_{r,m}$ is linear, it suffices to show that it is continuous at $x=0,$ by checking
\begin{eqnarray}\label{valpsi}
v(\D_{r,m}(x)) \geq [\frac{v(x)}{q^m}],
\end{eqnarray}
where $[a]$ is the greatest integer number $\leq a.$
Setting $n=v(x)$, write $x=T^ny$ with $(T,y)=1$ and
$n=lq^m +s$ with $0\leq s <q^m$ and $l\geq 0.$
Lemma \ref{basic} (2)  gives
\begin{eqnarray*} \label{valdrm}
\D_{r,m}(x)=\D_{r,m}(T^ny) =T^{[n/q^m]}\D_{r,m}(T^{s}y).
\end{eqnarray*}
Because $v((T^{s}y))\geq 0$, the preceding equality yields the desired inequality in (\ref{valpsi}).
\end{proof}


It is now  of great interest to find explicit expansions of $\D_{r,m}$ and its $q$th powers, in terms of Hasse derivatives.

\begin{theorem}\label{power1}
For  $t, m$ nonnegative integers and  $0 \leq r <q^m,$
$$ \D_{r,m}^{q^t} = \sum_{n=0}^{\infty}C_{r,n}^{(t)}\sD_{n},$$
where $C_{r,n}^{(t)}=(-1)^{n-r}\binom{s}{r}T^{s-r}(T^{q^m}-T^{q^t})^l$
if $n=lq^m +s$  with  $0 \leq s <q^m $ and $ l\geq 0.$
\end{theorem}

\begin{proof}
By the formula for coefficients in Theorem \ref{Hre}, we have
$$ C_{r,n}^{(t)} =\sum_{i=0}^{n}(-1)^{n-i}\binom{n}{i}T^{n-i}\D_{r,m}^{q^t}(T^i).$$
Writing $n=lq^m +s,$ $i=l'q^m +s'$ with  $0 \leq s,s' <q^m $ and $l, l' \geq 0,$ we obtain
\begin{eqnarray*}
 C_{r,n}^{(t)} &=& \sum_{l'=0}^{l}\sum_{s'=0}^{s}(-1)^{n-l'q^m -s'}T^{s-s'}T^{(l-l')q^m}\binom{lq^m +s}{l'q^m +s'}\D_{r,m}^{q^t}(T^{l'q^m +s'}) \\
 &=& \sum_{l'=0}^{l}(-1)^{n-l'q^m -r}T^{s-r}T^{(l-l')q^m}\binom{lq^m +s}{l'q^m +r}T^{q^t l'}.
\end{eqnarray*}
Because $\binom{lq^m +s}{l'q^m +r} =\binom{l}{l'}\binom{s}{r}$ in $\Fq$ from the Lucas congruence \cite{Lu}, we have
\begin{eqnarray*}
C_{r,n}^{(t)} &=& (-1)^{n-r}\binom{s}{r} T^{s-r}T^{lq^m}  \sum_{l'=0}^{l}(-1)^{l'}\binom{l}{l'}(T^{q^t-q^m})^{l'} \\
&=& (-1)^{n-r}\binom{s}{r} T^{s-r}(T^{q^m})^l (1-T^{q^t -q^m})^{l} \\
& =& (-1)^{n-r} \binom{s}{r}T^{s-r}(T^{q^m}-T^{q^t})^{l}.
\end{eqnarray*}
The proof is complete.
\end{proof}

From Theorem \ref{power1}, we select the most important case for which $t=m,$ which gives that
the right $q$th powers of $\D_{r,m}$ have a finite expansion in terms of the Hasse derivatives.

\begin{theorem}\label{singleout}
For $0 \leq r <q^m,$

$$ \D_{r,m}^{q^m} = \sum_{n=r}^{q^m -1}\binom{n}{r}(-T)^{n-r}\sD_{n}.$$
In particular,
$$ \D_{q^m -1,m}^{q^m} = \sD_{q^m-1}.$$

\end{theorem}

\begin{proof}
From Theorem \ref{power1},  $C_{r,n}^{(m)}$ vanishes only if $l \not =0$ in the expression of $n=lq^m +s.$ Hence,
$C_{r,n}^{(m)} =(-1)^{n-r}\binom{n}{r}T^{n-r}$ for $r \leq n< q^m.$
The  second assertion follows immediately from the first assertion.
\end{proof}

It is also of interest to find the inversion formula to the identity in Theorem \ref{singleout}.
This indicates that the resulting formula provides an alternative way of calculating ${\sD}_{n}(a)$ in terms of Cartier operators.

\begin{theorem}\label{inv}
For $0 \leq n <q^m,$
$$ {\sD}_{n}  = \sum_{r=n}^{q^m -1}\binom{r}{n}T^{r-n}\D_{r,m}^{q^m}.$$

\end{theorem}
\begin{proof}
Let $C(-T)$ be the $q^m$ by $q^m$ transition matrix from $ \{ {\sD}_{n} \}_{ 0\leq  n \leq q^m -1}$
into $ \{ \D_{r,m}^{q^m} \}_{ 0\leq  r \leq q^m -1}.$
From Theorem \ref{singleout}, $C(-T)$ is a matrix  whose $(n,r)$ entry is $C_{n,r}^{(m)}=\binom{n}{r}(-T)^{n-r}.$ In addition, it is an upper triangular matrix whose diagonal entries are all 1; thus, it is invertible.
We claim that the inverse of $C(-T)$ is $C(T)$ whose $(n,r)$ entry is $B_{n,r}^{(m)}=\binom{n}{r}T^{n-r},$ equivalently $C(T)C(-T)$ is the identity  matrix.
For this, we compute the $(n,r)$ entry, denoted $M_{n,r},$ of  $C(T)C(-T)$ as follows:
\begin{eqnarray*}
M_{n,r} &=& \sum_{k=0}^{q^m -1}B_{n,k}^{(m)}C_{k,r}^{(m)} \\
&=& T^{n-r}\sum_{k=0}^{q^m -1}\binom{n}{k}\binom{k}{r}(-1)^{n-k} \\
& =& \binom{n}{r}(-T)^{n-r}\sum_{k=r}^{n}(-1)^{k-r}\binom{n-r}{k-r}\\
& =& \binom{n}{r}(-T)^{n-r}(1-1)^{n-r} =\delta_{n,r},
\end{eqnarray*}
where $\delta_{n,r}$ is the Kronecker delta symbol, which is used extensively
from this point onwards.
Then the result follows.
\end{proof}

In Theorem \ref{gmain} we show that Theorems \ref{singleout} and \ref{inv} hold
over a field of Laurent series with coefficients in a more general field of characteristic $p>0.$
We are now ready to introduce two sets  of Cartier operators on $R=\Fq[[T]].$

\begin{definition}\label{ConR}
Let $n$ be an integer with $q^{k-1} \leq n <q^k,$ or $(n,k)=(0,0).$

(1) We define  a sequence of Cartier operators $\{\phi_{n}(x)\}_{n\geq 0}$ on $R$ given by
\begin{eqnarray*}\label{phin}
\phi_n(x)=\D_{n,k}^{q^k}(x)=\sum_{i\geq0}x_{iq^k +n}T^{iq^k}.
\end{eqnarray*}

(2) We define  a sequence of Cartier operators $\{\psi_{n}(x)\}_{n\geq 0}$ on $R$ given by
\begin{eqnarray*}\label{psin}
\psi_n(x)=\D_{n,k}(x)=\sum_{i\geq0}x_{iq^k +n}T^{i}.
\end{eqnarray*}
\end{definition}
From the definitions of both $\phi_n$ and $\psi_n~(n\geq1)$, it follows that they implicitly involve the positive integer $k,$ which is uniquely determined by $k = [\log_{q}n] +1,$; that is, the number of $q$-adic digits of $n.$
In what follows, the letter $k$ is not appended to the notation of the two Cartier operators.
Later on, it is shown that it is even easier to use $\phi_n$ than  $\psi_n.$ However,
the latter is dealt with by way of the former and the two sets of Cartier operators  $\phi_n$  and  $\psi_n$
play the same role in compositions of their $q$-adic extensions.
It is now worth noting that for $q^{k-1} \leq n < q^k$ and $x,y \in R,$
\begin{eqnarray}
\phi_n &=& \psi_n^{q^k}, \nonumber \\
\phi_{n}(x^{q^k}y) &=& x^{q^k}\phi_{n}(y), \\ \label{id1R}
\phi_{q^k-1}(x) &=& \cD_{q^k -1}(x), \\ \label{id2R}
\phi_{n}(x) &=& \phi_{q^k-1}(T^{q^k-1-n}x), \\ \label{id3R}
\psi_n(x^{q^k}y)&=& x\psi_n(y).\label{id4R}
\end{eqnarray}

For $m =lq^k +m_0$ with $0 \leq m_0 <q^k$ and $l\geq 0,$ we have
\begin{equation}\label{pntm2}
 \phi_n(T^m)=T^{m-m_0}\delta_{n,m_0}~~ {\rm and}~~ \psi_n(T^m)=T^{[m/q^k]}\delta_{n,m_0}.
\end{equation}

For $n$, a positive integer written  in $q$-adic form as
\begin{eqnarray*}\label{nqadic}
n =\sum_{i=0}^{k-1}n_iq^i ~(0 \leq n_i<q ~{\rm and} ~n_{k-1} \not =0 ),
\end{eqnarray*}
we define $q(n)$ and $n_- $, respectively, as
\begin{equation}\label{q(n)}
q(n) =n_{k-1}q^{k-1}  ~{\rm and} ~ n_- = n-q(n).
\end{equation}

As one of the main results, we state the following theorem.
\begin{theorem}\label{pnrep}
Let $f: R \rightarrow K$ be an $\Fq$-linear continuous function.
Set
\begin{eqnarray}\label{cn}
c_0 &= & f(1),\nonumber \\
c_n&= & f(T^n) -T^{q(n)}f(T^{n_-})~(n\geq1),
\end{eqnarray}
where $q(n)$ and $n_-$ are defined in (\ref{q(n)}).
Then, $\{ \ph_n \}_ {n\geq 0}$
is an orthonormal basis for  $LC(R,K).$
That is, $\sum_{n=0}^{\infty}c_n \ph_n(x)$ converges uniformly to $f(x).$
\end{theorem}

\begin{proof}
We provide two proofs.
The first proof is based on standard arguments in non-Archimedean analysis, for example, \cite[Theorem, p. 183]{Ro}.

Because $f(x)$ is continuous at $x=0,$ it is obvious from (\ref{cn}) that
$c_n$ is a null sequence in $K$ because, as $n \rightarrow \infty,$ $q(n)\rightarrow \infty.$
As $R$ is compact, the series $\sum_{n=0}^{\infty}c_n \ph_n(x)$ converges uniformly to
the continuous function $f(x).$ Here we need to show that the two continuous functions are equal on $R.$
By $\Fq$-linearity and
continuity, it now suffices to check that they agree on all monomials $T^m(m\geq0).$
For this, write a positive integer $m=\sum_{j=0}^{k}m_{i_j}q^{i_j}$ in $q$-adic form such that
$m_{i_j} \not =0$ for all $0 \leq j \leq k$ and $ 0 \leq i_{0} < i_{1}< \cdots < i_{k}.$

Using (\ref{pntm2}) we calculate
\begin{eqnarray*}
 \sum_{n\geq0}c_n \ph_n(T^m) &=&  \sum_{n=0}^{m}c_n \phi_n(T^m)\\
&=& f(1)T^m +c_{m_{i_0}q^{i_0}}T^{m-m_{i_0}q^{i_0}} +c_{m_{i_0}q^{i_0} +m_{i_1}q^{i_1}}T^{m-m_{i_0}q^{i_0}-m_{i_1}q^{i_1}}
 + \cdots  \\
& &
  \cdots + c_{ m_{i_0}q^{i_0} + m_{i_1}q^{i_1}+\cdot+ m_{i_{k-1}}q^{i_{k-1}}}T^{m_{i_k}q^{i_k}}  +c_m.
\end{eqnarray*}
 By the formula in (\ref{cn}), the right hand side of the preceding equality
 equals
 $$ f(1)T^m +(f(T^{m_{i_0}q^{i_0}})-T^{m_{i_0}q^{i_0}}f(1))T^{m-m_{i_0}q^{i_0}}
 $$
 $$+(f(T^{m_{i_0q^{i_0} +m_{i_1}q^{i_1}}})-T^{q(m_{i_0}q^{i_0}+{m_{i_1}q^{i_1}})}f(T^{m_{i_0}q^{i_0}}))T^{m-m_{i_0}q^{i_0}-m_{i_1}q^{i_1}} + \cdots $$
 $$
 + \cdots + (f(T^{\sum_{j=0}^{k-1}m_{i_j}q^{i_j}})-T^{m_{i_{k-1}}q^{i_{k-1}}}f(T^{{\sum_{i=0}^{k-2}m_{i_j}q^{i_j}}})) T^{m_{i_k}q^{i_k}}  + f(T^m) -T^{q(m)}f(T^{m_-}). $$
It follows that the sum above becomes $f(T^m)$, because it is a telescoping sum, by (\ref{q(n)}).

In the usual way, from  (\ref{cn}), we deduce that $\|f\|=\mbox{max}\{|c_{n}|\}.$
For an alternative proof, we use \cite[Lemme I]{Se}, as was done with the Hasse derivatives in \cite{Co2} and \cite{J1,J6}.
In this lemma, a necessary and sufficient
condition for $\{ \ph_n \}_ {n\geq 0}$ to be an orthonormal basis for $LC(R,K)$
is that
\par
(1) $\ph_n $ maps $R$ into itself;
\par
(2) the reduced functions modulo $T$, denoted
$\overline{\ph_n },$ form  a basis for $LC(R,{\F}_q)$ as an ${\F}_q$-vector space.

Since Part (1) is trivial, we only check Part (2) by showing that for any integer $n>0$,
the reduced functions $\overline{\ph_0 }, \overline{\ph_1 } \cdots, \overline{\ph_{n-1} } $
are linearly independent in the $\Fq$-dual space $({\F}_{q}[T]/T^{n})^*.$ In fact, using (\ref{pntm2}), these functions  are
the dual basis to $1, T, \cdots, T^{n-1}.$

\end{proof}

The $q$th power maps have an explicit expansion in terms of the Cartier operators $\phi_n.$

\begin{corollary}\label{qmap}
For any integer $m\geq 0,$
$$x^{q^m} = x+ \sum_{n\geq1} (T^{nq^m}-T^{q( n)+q^mn_-})\phi_{n}(x).$$
\end{corollary}
\begin{proof}
The proof is immediate from Theorem \ref{pnrep}.
\end{proof}
The expansions in Corollary \ref{qmap} can be compared with Voloch's expansions of the $q$th power maps in terms of the Hasse derivatives $\cD_n$ (see \cite{V}):
$$x^{q^m} = \sum_{n\geq0}(T^{q^m} -T)^n\cD_{n}(x).$$

\begin{corollary}\label{qmdelta}
For $q^{k-1} \leq r <q^k \leq q^m$ or $(r,k)=(0,0)$,
$$\D_{r,m}^{q^m}(x)= \phi_r(x)-\sum_{\substack{k \leq i \leq m-1 \\ 1 \leq j \leq q-1 }}T^{jq^i}\phi_{jq^i +r}(x).$$
\end{corollary}
\begin{proof}
By Theorem \ref{pnrep}, writing $\D_{r,m}^{q^m}(x)=\sum_{n\geq0}  B_{r,n}^{(m)}\phi_{n}(x),$  we have $B_{r,0}^{(m)}=0$
and  for $n\geq1,$
\begin{eqnarray}\label{Bmrn}
B_{r, n}^{(m)} =\D_{r,m}^{q^m}(T^n)-T^{q(n)}\D_{r,m}^{q^m}(T^{n_-}).
\end{eqnarray}
If $n\geq q^m,$ then
writing $n=lq^m +s$ with $l\geq 1$ and $0 \leq s <q^m$ we have
$$ B_{r, n}^{(m)} =T^{n-s}\d_{r,s}-T^{q(n)+n_- -s}\d_{r,s}=0.$$
For $r > n,$ $ B_{r, n}^{(m)}$ also vanishes; thus, we may assume that $ r \leq n <q^m.$
For this case, we have
$$ B_{r, n}^{(m)} =\d_{r,n}-T^{q(n)}\d_{r,n_-}.$$
From this relation we deduce that
$B_{r, n}^{(m)}=1$ if $n=r$,  $-T^{q(n)}$ if $n>r$, and  $n_- =  r,$ 0 if $ n>r$ and $n_- = r.$
This case shows that $n$ is of the  form $n =jq^{i}+ r$ with $k \leq i <m$ and $1\leq j \leq q-1.$
The proof of the case where $r=0=k$ follows in the similar way.
\end{proof}

The following corollary can be deduced in a similar fashion as Corollary \ref{qmdelta}.

\begin{corollary}\label{qmdelta2}
For $q^{k-1} \leq r <q^k \leq q^m$ or $(r,k)=(0,0)$,
$$\D_{r,m}(x)= \phi_r(x)+\sum_{j>0}(T^j -T^{jq^m +j_-})\phi_{jq^m+r}(x) -\sum_{\substack{k \leq i \leq m-1 \\ 1 \leq j \leq q-1 }}T^{jq^i}\phi_{jq^i +r }(x).$$
\end{corollary}




We state another main result related to the Cartier operators $\psi_n.$
\begin{theorem} \label{car2}
$\{ \psi_n \}_ {n\geq 0}$
is an orthonormal basis for  $LC(R,K).$
\end{theorem}

\begin{proof}
We provide two proofs. For the first proof we invoke the following from Corollary \ref{qmdelta2} with $m=k:$
For $q^{k-1} \leq n <q^k,$
$$\psi_n(x)= \phi_n(x)+ \sum_{j>0}(T^j -T^{jq^k +j_-})\phi_{jq^k +n}.$$
This identity implies that $\psi_n \equiv \phi_{n} \pmod{T},$;  that is,
$||\psi_n - \phi_{n} || <1$ for all $n \geq 0.$ Then, the result follows
from Lemma \ref{ba}, together with Theorem \ref{pnrep}.

The second proof follows immediately from the same argument applied to  the second proof of Theorem \ref{pnrep}.
\end{proof}

Unlike $\ph_n,$ it is not easy to find a closed-form formula for coefficients in the representation of
$$f =\sum_{n\geq 0}B_n\psi_n(x) \in LC(R,K).$$
However, if $f \in LC(R,K)$ assumes values in $R,$ it is very useful to observe the
following simple relation on coefficients $B_n$ from the proof of Theorem
\ref{pnrep}. For all $n\geq1,$
\begin{eqnarray}\label{coeffps}
B_n \equiv f(T^n) \pmod{T}.
\end{eqnarray}
This relation is sufficient to be used for proving that another sequence of operators is
an orthonormal basis for  $LC(R,K)$ in the next subsection.
Refer to the formula for $B_n$ in (\ref{psifor}), which is indirectly derived from the coefficient formula for $f \in C(R,K).$

As for composition we have the following results.
\begin{proposition}
(1) For $q^{k-1}\leq m,n <q^k,$
$\phi_n \circ \phi_m =0 = \phi_m \circ \phi_n .$

(2) For $q^{k-1}\leq n <q^k \leq q^{l-1} \leq m < q^l,$

(a) $\phi_n \circ \phi_m =0 ,$

(b) $$\phi_m \circ \phi_n  = \sum_{\substack{ i\geq 0 }}T^{iq^{l}}\phi_{m+n+iq^{l}}. $$
\end{proposition}

\begin{proof}
For (1), because $\phi_m(x)$ and $\phi_n(x)$ are $q^k$th powers, we have, by (\ref{id1R}) and (\ref{id4R}),
$$  \phi_n \circ \phi_m(x)=  \phi_n (\phi_m(x))=  \phi_m(x) \phi_n (1)=0.$$
Similarly, for $ \phi_m \circ \phi_n =0$ as well as (2)-(a).
Now for (2)-(b), by Theorem \ref{pnrep}, write $\phi_m \circ \phi_n =\sum_{j\geq 0}B^{(m,n)}_j\phi_j(x).$
Then, $B^{(m,n)}_0 =0$ and for $j>0$ $$B^{(m,n)}_j =\phi_m \circ \phi_n(T^j) -T^{q(j)}\phi_m \circ \phi_n(T^{j_-}).$$
Writing $j=sq^k +r$ with $0 \leq r < q^k$ and $s\geq 0,$
we have
\begin{equation*}
B^{(m,n)}_j= \left \{ \begin{array}{ll}
 \phi_n(T^{l -n})-T^{q(l)}\phi_n(T^{l_- -n})   & \mbox{if}~~ n=r
\\ 0   & otherwise.
 \end{array}
 \right.
\end{equation*}
Writing  $j-n=s'q^l +r'$ with $0 \leq r' < q^l$ and $s'\geq 0,$ if $r \leq q(j)$ (if necessary) then
$$ B^{(m,n)}_j =T^{j -n-r}\delta_{m,r'}-T^{q(j)+j_-n -(r'-q(j))}\delta_{m,r'-q(j)}. $$
The cases in which  $B^{(m,n)}_j$ may have non-zero coefficients  are those for which
$(\delta_{m,r'}, \delta_{m,r'-q(j)})$=(1,0) or (0,1).
For the other case (0,1),  $j$ is of the  form $j=m+n+s'q^l +
q(j).$ From this, $j_- =m+n+s'q^l < q(j) \leq  r' <q^l$ so we have $s'=0.$
As $q^{l-1}\leq q(j)\leq r'<q^{l}$, we have $q(j)=iq^{l-1}$ with $ 1\leq i <q.$
Plugging this $q(j)$ into the equation $j=m+n+q(j)$ has a contradiction, completing the proof.
\end{proof}

The product rule can be stated as follows.

\begin{lemma}
For $q^{k-1}\leq n <q^k$ and $x, y \in R,$
\begin{eqnarray*}
 \phi_n (xy) &=& \phi_n(y)\phi_{0}(x) +\sum_{j=1}^{q^k-1} (\phi_n(T^j y)-T^{q( j)}\phi_ n(T^{j_-} y))\phi_{j}(x) \\
&=& \phi_n(x)\phi_{0}(y) +\sum_{j=1}^{q^k-1} (\phi_n(T^j x)-T^{q( j)}\phi_ n(T^{j_-} x))\phi_{j}(y).
\end{eqnarray*}
\end{lemma}
\begin{proof}
We only prove the first identity because the interchanging roles of $x$ and $y$ provide the second identity.
As an $\Fq$-linear continuous operator of $x$, write
$\phi_n (xy) =\sum_{j=0}^{\infty}C_j(y)\phi_j(x).$ Then, by the formula in Theorem \ref{pnrep},
$C_j(y)= \phi_n(T^j y)-T^{q( j)}\phi_ n(T^{j_-} y)$ for $j>0$  and  $C_0(y)= \phi_n(y).$
Now, it is easy to verify that $C_j(y)$ vanishes for  all $j\geq q^k.$
\end{proof}

\subsection{Relations among the five bases}
Thus far, we introduced five orthonormal bases for $LC(R,K)$ among which two are new.
Here we show the existence of a close relation between any two orthonormal bases among these five bases.
Indeed, Theorem \ref{eqlinear} states that any two bases are equivalent to each other, which means
that if one base is an orthonormal basis for $LC(R,K)$, so is the other and vice versa.
Each lemma below contains an inverse relation between any two orthonormal bases if either one is known to be orthonormal.

The following Lemma \ref{bif} shows that the Hasse derivatives and Cartier operators satisfy
the binomial inversion formula, which is extended to Theorem \ref{gmain} in a complete generality.
\begin{lemma}\label{bif}
For $q^{k-1}\leq  n< q^k,$ or $(n,k)=(0,0),$
\begin{itemize}
\item[] {\rm (1)} $ \cD_{n}(x) =\sum_{m=n}^{q^k -1} \binom{m}{n}T^{m-n}\ph_m(x).$
\item[] {\rm (2)} $ \ph_n(x) =\sum_{m= n}^{q^k -1} \binom{m}{n}(-T)^{m-n} \cD_{m}(x).$
\end{itemize}
\end{lemma}

\begin{proof}
By Theorem \ref{pnrep}, write $\cD_n(x) =\sum_{m\geq0} B_{n,m}\phi_n(x),$ and then
$$B_{n,m} =\binom{m}{n}T^{m-n}-T^{q(m)}\binom{m_-}{n}T^{m_- -n}~(n>0).$$
From this we note that $B_{n,m}=0$ for $m<n$ and $B_{n,m}=1$ for $m=n.$ If $m\geq q^k$, then $B_{n,m}=0,$ by Lucas's congruence.
If $ n \leq m < q^k$, then $ m_- < q^{k-1},$ such that
 $B_{n,m}= \binom{m}{n}T^{m-n}.$  As the case where $(n,k)=(0,0)$ is obvious, this completes the proof of Part (1).
Part (2) is a restatement of Theorem \ref{inv}.
\end{proof}

\begin{lemma}
For $q^{k-1}\leq  n< q^k,$ or $(n,k)=(0,0),$
\begin{itemize}
\item[] {\rm (1)} $ \cD_{n}(x) =\sum_{m\geq 0} B_{n,m}\psi_m(x),$
where $B_{n,m}  \equiv 1 \pmod{T}$ if $n=m$; otherwise, $B_{n,m} \equiv 0 \pmod{T}$.
\item[] {\rm (2)}
$ \psi_n(x) =\sum_{m\geq0} C_{n,m}\cD_{m}(x), $
where $C_{n,m}=(-1)^{m-n}\binom{s}{r}T^{s-r}(T^{q^k}-T)^l$
if $m=lq^k +s$  and $0 \leq s <q^k.$
\end{itemize}
\end{lemma}

\begin{proof}
Part (1) follows from Theorem \ref{car2} and the relation in (\ref{coeffps}).
Part (2) is a restatement of Theorem \ref{power1}, observing
that $C_{n,m}=0$ for $m<n,$ $C_{n,m}=1$ for $m=n,$ and $C_{n,m} \equiv 0 \pmod{T}$ for $m>n.$
\end{proof}

\begin{lemma}
For $n\geq0,$
\begin{itemize}
\item[] {\rm (1)}
$ \psi_{n}(x) =\sum_{m\geq 0} B_{n,m}\phi_m(x),$
where $B_{n,m}  =\psi_n(T^m) -T^{q(m)}\psi_n(T^{m_-})~ and~B_{n,0}=\delta_{n0}.$
\item[] {\rm (2)}
$ \ph_{n}(x) =\sum_{m\geq 0} C_{n,m}\psi_m(x),$
where $C_{n,m}  \equiv 1 \pmod{T}$ if $n=m$; otherwise, $C_{n,m} \equiv 0 \pmod{T}$.
\end{itemize}
\end{lemma}

\begin{proof}
From Theorem \ref{pnrep} we deduce that $B_{n,m}=1$ if $m=n$ and $B_{n,m} = 0 $ if $m <n.$
For $q^{k-1} \leq n <q^k,$
if $m=lq^k +n$ for $l>0$ then $B_{n,m} =T^l -T^{m-n+l_-} \equiv 0\pmod{T}.$
If $m\not \equiv n \pmod{q^k}$ with $m >n,$ then $B_{n,m}\equiv 0 \pmod{T}.$
Hence, $B_{n,m} \equiv 0 \pmod{T} $ for $m>n.$
Part (2) follows from Theorem \ref{car2} and the relation in (\ref{coeffps}).
\end{proof}

\begin{lemma}
For any $n\geq 0,$
\begin{itemize}
\item[] {\rm (1)} $ E_n(x) =\sum_{m\geq0} B_{n,m} \phi_m (x),$
where $B_{n,m}\equiv 1 \pmod{T}$ if $m=n,$; otherwise,  $B_{n,m} \equiv 0 \pmod{T}$.
\item[] {\rm (2)} $ \phi_n(x) =\sum_{m\geq0} C_{n,m}E_m (x),$
where $C_{n,m}\equiv 1 \pmod{T}$ if $m=n$; otherwise, $C_{n,m} \equiv 0 \pmod{T}$.
\end{itemize}
\end{lemma}

\begin{proof}
Parts (1) and (2) follow from Theorem \ref{pnrep} and Theorem (\ref{Wre}), respectively.
\end{proof}

\begin{lemma}
For any $n\geq 0,$
\begin{itemize}
\item[] {\rm (1)} $ \psi_m (x)  =\sum_{{m\geq0}}  B_{n,m}E_n(x),$
where $B_{n,m}\equiv 1 \pmod{T}$ if $m=n$; otherwise, $B_{n,m}\equiv 0 \pmod{T}$.
\item[] {\rm (2)} $ E_n(x)  =\sum_{{m\geq0}} C_{n,m}\psi_m,$
where $C_{n,m}\equiv 1 \pmod{T}$ if $m=n$; otherwise, $C_{n,m}\equiv 0 \pmod{T}$.
\end{itemize}
\end{lemma}

\begin{proof}
Parts (1) and (2) follow from Theorem \ref{Wre} and Theorem \ref{car2} and the relation in (\ref{coeffps}), respectively.
\end{proof}

\begin{lemma}\label{pnsn}
\begin{itemize}
\item[] {\rm (1)} For any $n\geq 0,$
 $ \S^{(n)}(x) =\sum_{m_- <n} T^{m-n} \phi_m(x).$
\item[] {\rm (2)} For $q^{k-1}\leq  n< q^k,$ or $(n,k)=(0,0),$ \\
 $ \phi_n(x)  =\sum_{i\geq0}T^{iq^k}\S^{(iq^k)}(x) -\sum_{i\geq0}T^{iq^k+1}\S^{(iq^k+1)}(x).$
\end{itemize}
\end{lemma}
\begin{proof}
From Theorem \ref{pnrep}, writing $\S^{(n)}(x) =\sum_{m\geq0} B_{n,m}\phi_n(x),$
we have $$B_{n,m} =\S^{(n)}(T^m) -T^{q(m)}\S^{(n)}(T^m)~(m\geq 1).$$
It is easy to see that if $m>n$ and $m_- \leq n$ then $B_{n,m}=0$ and if $m>n$ and $m_- <n$
then $B_{n,m}=T^{m-n}.$ Hence, we have the desired result.
By Theorem \ref{snrep}, writing $\phi_n(x)  =\sum_{m\geq0}B_{n,m}\S^{(n)}(x),$ we obtain
$$B_{n,m} =\phi_{n}(T^m) -T\phi_n(T^{m-1})~(m\geq0)~~B_{n,0}=\delta_{n,0}.$$
Setting $m=lq^k +r$ with $0 \leq r <q^k$ and $l\geq0,$
we have $$B_{n,m}=T^{m-r}\delta_{m,r}-T^{m-r+1}\delta_{n,r-1},$$
giving the desired result.
\end{proof}

\begin{lemma}
\begin{itemize}
\item[] {\rm (1)} For any $n\geq 0,$
 $ \S^{(n)}(x) =\sum_{m\geq0} B_{n,m} \psi_n(x),$ \\
where $B_{n,m}\equiv 1 \pmod{T}$ if $m=n$; otherwise, $B_{n,m}\equiv 0 \pmod{T}$.
\item[] {\rm (2)} For $q^{k-1}\leq  n< q^k$ or $(n,k)=(0,0),$ \\
 $\psi_n(x)  =\sum_{i\geq0}T^{i}\S^{(iq^k)}(x) -\sum_{i\geq0}T^{i+1}\S^{(iq^k+1)}(x).$
\end{itemize}
\end{lemma}

\begin{proof}
Part (1) follows from Theorem \ref{car2} and the relation in (\ref{coeffps}).
The proof of (2) is similar to that of (2) in Lemma \ref{pnsn}.
\end{proof}

\begin{theorem}\label{eqlinear}
The following are equivalent:
\begin{itemize}
\item[] {\rm (1)} $\{E_{n}\}_ {n\geq 0}$
is an orthonormal basis for $LC({R},K).$
\item[] {\rm (2)}  $\{\cD_n\}_ {n\geq 0}$
is an orthonormal basis for  $LC({R},K).$
\item[] {\rm (3)}  $\{\S^{(n)}\}_ {n\geq 0}$
is an orthonormal basis for  $LC({R},K).$
\item[] {\rm (4)}  $\{\ph_n\}_ {n\geq 0}$
is an orthonormal basis for  $LC({R},K).$
\item[] {\rm (5)} $\{\psi_n\}_ {n\geq 0}$
is an orthonormal basis for  $LC({R},K).$
\end{itemize}
\end{theorem}

\begin{proof}
All equivalences follow from Lemma \ref{ba} and Lemmas above, such that  the detailed proofs are omitted here.
\end{proof}

\section{Orthonormal bases for $C(\O,\V).$}
Now, we redirect our attention to continuous functions from $\O$ to $\V$ where $\O$ is the integer ring of a local field $\V$ as in Section 2.1. We provide two sets of orthonormal bases of the entire space $C(\O,\V)$ of all continuous functions from $\O$ to $\V$ for two distinguished cases where $\O=R$ or $\O=\Z_p.$ These results essentially follow from Conrad's digit principle: \cite[Theorem 2]{Co2} for $\Fq[[T]]$ and \cite[Theorem 11]{Co2} for $\Zp.$

\subsection{Two $q$-adic digit Cartier bases of $C(R,K).$}
The following definition is crucial for the construction of an orthonormal basis for  $C(R,K)$ out of that of $LC(R,K).$
\begin{definition}
Let $\{f_{i}(x)\}_ {i\geq 0}$ be an orthonormal basis for $LC({R},K),$ and let
$$ n = {n}_{0}+{n}_{1}q + \cdots +{n}_{w-1}q^{w-1}$$
be the $q$-adic expansion of  any integer $n\geq 0,$ with $ 0\leq {n}_{i} < q.$ Set
\begin{equation*}
\CF_n(x) :=\prod_{i=0}^{w-1}f_{i}^{n_i}(x)~(n\geq 1),~\CF_0(x) =1,
\end{equation*}
and
\begin{equation*}\CF_{n}^{*}(x) :=\prod_{i=0}^{w-1}\CF_{n_iq^i}^{*}(x)~(n\geq 1),~\CF_{0}^{*}(x) =1,
\end{equation*}
where
\begin{equation*}
\CF_{n_iq^i}^{*}(x) = \left \{ \begin{array}{ll}
f_{i}^{q-1}(x)-1     & \mbox{if}~~ n_i=q-1
\\ f_{i}^{n_i}(x)   & \mbox{if}~~
 n_i<q-1.
 \end{array}
 \right.
\end{equation*}
Then, we say $\{\CF_{n}(x)\}_ {n\geq 0}$ is a $q$-adic extension of $\{f_{i}(x)\}_ {i\geq 0}$ in $LC({R},K).$
\end{definition}

Well-known examples of such $q$-adic extensions include
\begin{equation*}\label{threeb4}
(f_i,\CF_n,\CF_n^{*})= (E_i,G_n,G_n^{*}),(\cD_{i},{\sf D}_{n},{\sf D}_n^{*}), (\S^{(i)},\S_{n},\S_{n}^{*}),
\end{equation*}
where $E_i, \cD_{i}$, and $\S^{(i)}$ are referred to as Carlitz linear polynomials, Hasse derivatives, and shift operators, respectively, as in Section 2.1,
and $G_n, {\sf D}_{n},$, and $\S_{n}$ are referred to as Carlitz polynomials, digit derivatives, and digit shifts, respectively.
It is shown in \cite{Ca} that Carlitz polynomials $G_n$ are a prototypal $q$-adic extension of $E_i.$
Besides these examples, by Theorems \ref{pnrep} and \ref{car2} we add to this list two more $q$-adic extensions
\begin{equation}\label{twoCs}
(f_i,\CF_n,\CF_n^{*})= (\phi_i,\Phi_n,\Phi_n^{*}),(\psi_i,\Psi_n,\Psi_n^{*}),
\end{equation}
where $\Phi_n$ and $\Psi_n $ are referred to as digit Cartier functions.
For the remainder of this subsection, we only consider $(f_i,\CF_n,\CF_n^{*})$ in (\ref{twoCs}) to emphasize
both of these digit Cartier functions.
We now examine some properties of these two functions on $R$ as compared to constructions of other $q$-adic extensions.
All properties are modeled on properties such as the binomial and orthogonal properties of Carlitz
polynomials. Note that
\begin{equation}\label{threeb5}
f_{i}(T^j) = \left \{ \begin{array}{ll}
0 & \mbox{if}~~ j<i \\
1 & \mbox{if}~~ j=i   \\
\equiv 0 \pmod{T} & \mbox{if}~~ j>i.
\end{array} \right.
\end{equation}


\begin{proposition}\label{pG}
The binomial formulas for $\CF_n$  and $\CF_n^{*}$ are
\begin{itemize}
\item[] {\rm(1)} $\CF_{n}({\l}x)= {\l}^{n}\CF_{n}(x)$  for $\l \in {\F}_{q}^*.$
\item[]{\rm(2)} $\CF_{n}(x+y)= \sum_{i=0}^n \binom{n}{i}\CF_{i}(x)\CF_{n-i}(y).$
\item[] {\rm(3)} $\CF_{n}^{*}({\l}x) = {\l}^{n}\CF_{n}(x)$  for $\l \in {\F}_{q}^*.$
\item[] {\rm(4)} $\CF_{n}^{*}(x+y) = \sum_{i=0}^n \binom{n}{i}\CF_{i}(x)\CF_{n-i}^{*}(y).$
\end{itemize}
\end{proposition}
\begin{proof}
The proof follows by adopting the arguments in \cite{Ca} or \cite{Go}  for $G_n$  and $G_n^{*}$ to our case.
\end{proof}

Because $\binom{q^{m} -1}{i} = (-1)^{i}$ in ${\F}_q,$ we obtain  the
following corollaries of Propositions \ref{pG}.
\begin{corollary}
\noindent {\rm(1)} $\CF_{q^{m}-1}(x+u)=
\sum_{i+j=q^{m}-1}(-1)^{i}\CF_{i}(x)\CF_{j}(u).$

\noindent {\rm(2)} $\CF_{q^{m}-1}(x-u)=
\sum_{i+j=q^{m}-1}\CF_{i}(x)\CF_{j}(u). $

\noindent {\rm(3)}
$\CF_{q^m-1}^{*}(x+u)=\sum_{i+j=q^{m}-1}(-1)^{i}\CF_{i}(x)\CF_{j}^{*}(u).$

\noindent {\rm(4)} $\CF_{q^{m}-1}^{*}(x-u)=
\sum_{i+j=q^{m}-1}\CF_{i}(x)\CF_{j}^{*}(u). $
\end{corollary}

The following is the orthogonality property of the two digit Cartier functions.
\begin{proposition}\label{dorg}
\par
{\rm(1)} For $l <q^n, k$ an arbitrary integer $\geq0,$
\[ \sum_{\scriptsize{\begin{array}{l}
                 \a \in {\F}_{q}[T] \nonumber \\
                 \mbox{deg}(\a)< n \nonumber
                   \end{array}}}{{\CF}_{k}(\a){\CF}_{l}^{*}(\a)}   =
                   \left \{ \begin{array}{ll}
0       & if k+l \not =q^n-1 \\
 (-1)^n   & if k+l =q^n-1.
 \end{array}
 \right. \]

\par
{\rm(2)} For $l <q^n, k < q^n,$
\[ \sum_{\scriptsize{\begin{array}{l}
                 \a \hspace{0.1in} monic \nonumber \\
                 \mbox{deg}(\a) = n \nonumber
                   \end{array}}}{{\CF}_{k}(\a){\CF}_{l}^{*}(\a)}
                   = \left \{ \begin{array}{ll}
                   0       & if k+l \not =q^n-1 \\
 (-1)^n   & if k+l =q^n-1.
 \end{array}
 \right. \]
\end{proposition}

\begin{proof}
Two proofs are known for the case $(\CF_n, \CF_n^{*})=(G_n, G_n^{*}).$  Indeed, Carlitz \cite{Ca} provided the original proof, which is based on interpolations of his polynomials. Yang \cite{Ya} established the same result in a direct yet elementary way. These two arguments are also applied to the two digit Cartier functions, as in the case of digit derivatives  and digit shifts \cite{J1,J3}. In particular, Yang's argument operates with the digit principle once we have a basis $\{f_i\}_{i\geq0}$ for $LC(R,K),$ having additional property in (\ref{threeb5}).
\end{proof}

The digit principle leads to the following main result.
\begin{theorem}\label{sjrep}
\quad
Let $(\CF_n,\CF_n^{*})= (\Phi_n,\Phi_n^{*})$ or $(\Psi_n,\Psi_n^{*}).$
\begin{itemize}
\item[] {\rm (1)} $\{ \CF_n(x)\}_{n\geq 0}$ is an orthonormal basis for $C(R,K).$
\item[] {\rm (2)} Write $f \in C(R,K)$ as $ f(x)=\sum_{n\geq0}c_n \CF_n(x).$ Then, for any integer $w$ such that $q^w >n$, $c_{n}$ can be recovered by
 $$c_{n}
=(-1)^{w}\sum_{\a \in A_{w}}\CF_{q^w-1-n}^{*}(\a)f(\a),$$
where $ A_{w}$ denotes the set of all polynomials in $T$ with coefficients in $\Fq$ of degree $<w.$
\end{itemize}
\end{theorem}

\begin{proof} Part (1) follows from applying the digit principle to two bases, $\phi_n$ and $\psi_n$
 for $LC(R,K)$ in Theorems \ref{pnrep} and \ref{car2}.
For part (2), it follows from the orthogonality property in Proposition \ref{dorg}.
For any integer $w$ such that $q^w >n$, we have $$
(-1)^w \sum_{\a \in A_w}{\CF}_{q^w-1-n}^{*}(\a)f(\a)
=\sum_{j=0}^{\infty}c_{j}(-1)^{w}\sum_{\a \in A_w}{\CF}_{q^w-1-n}^{*}(\a){\CF}_j(\a)=c_n.$$
\end{proof}


Application of the digit principle to Theorem \ref{eqlinear} produces the following result.

\begin{theorem} The following are equivalent:
\begin{itemize}
\item[] {\rm (1)} $\{G_{n}\}_{n\geq 0}$ is an orthonormal basis for $C({R},K).$
\item[] {\rm (2)} $\{{\sf D}_{n}\}_{n\geq 0}$ is an orthonormal basis  for   $C({R},K).$
\item[] {\rm (3)} $\{\S_{n}\}_{n\geq 0}$ is an orthonormal basis  for   $C({R},K).$
\item[] {\rm (4)} $\{\Phi_{n}\}_{n\geq 0}$ is an orthonormal basis  for   $C({R},K).$
\item[] {\rm (5)} $\{\Psi_{n}\}_{n\geq 0}$ is an orthonormal basis  for   $C({R},K).$
\end{itemize}
\end{theorem}

\begin{proof}
The proof follows from the application of the digit principle to Theorem \ref{eqlinear}.
For the individual proofs of the first three statements we refer the reader to the following work:
see \cite{W1,Go, Co2, Ya} for Part (1) and \cite{Sn,J1, J3,J4, Co1} for Part (2)
and \cite{J6} for Part (3).
\end{proof}

For $f \in C(R,K)$ to lie in $LC(R,K),$ we provide the conditions in terms of its coefficients.
\begin{corollary}\label{ch}
Write $f \in C(R,K)$ as $f(x) = \sum_{n=0}^{\infty} c_{n}
\CF_{n}(x).$  Then $f \in LC(R,K)$ if and only if $c_{n} =0$ for
$n \not =q^{i},$ where $i\geq0.$
\end{corollary}

\begin{proof}
This follows from Theorem \ref{eqlinear}. However, an alternative proof follows by adopting the arguments \cite{W1} or \cite{J3})
to our case.
\end{proof}

From Corollary \ref{ch} and Theorem \ref{sjrep}  we can indirectly retrieve the coefficients of $f(x) = \sum_{n=0}^{\infty} B_{n}\psi_{n}(x)$ by computing, for any $w$ such that $q^n <q^w,$
\begin{eqnarray}\label{psifor}
B_{n}=(-1)^{w}\sum_{\a \in A_{w}}\Psi_{q^w-1-q^n}^{*}(\a)f(\a).
\end{eqnarray}
The formula for the coefficients $c_n$ in Theorem \ref{sjrep} yields the following corollary.
\begin{corollary}
Let $f(x) =\sum_{n\geq0}c_n{\CF}_{n}(x)$ be a continuous
function from $R$ to $K.$ Then $f(x) \in C(R,R)$ if and only if
$\{c_n\}_{n\geq0} \subset R.$
\end{corollary}

\subsection{Two $p$-adic digit Cartier bases of $C(\Z_p, \Q_p)$}

Here, we introduce analogues in $\Zp$ of two Cartier operators \
and then show that  $p$-adic extensions of these analogues form an
orthonormal basis of $C(\Z_p, \Q_p)$.
Now, the Cartier maps on $\Z_p$ can be defined in the same way as was with $R$ in Definition \ref{ConR}, with
the same notation to denote those maps.
\begin{definition}\label{ConZ}
Let $n$ be an integer such that $p^{k-1} \leq n <p^k,$ or $(n,k)=(0,0).$

(1) The Cartier map $\phi_n$ on $\Z_p$ is defined by
$$ \phi_n(\sum_{i\geq0}x_ip^i)=\sum_{i\geq0}x_{ip^k+n}p^{ip^k}.$$

(2)  The Cartier map $\psi_n$ on $\Z_p$ is defined by
$$ \psi_n(\sum_{i\geq0}x_ip^i)=\sum_{i\geq0}x_{ip^k+n}p^{i}.$$

\end{definition}

\begin{lemma}
For each $n\geq 0,$ $\phi_n$ is a continuous function on $\Z_p$ and so is $\psi_n.$
\end{lemma}
\begin{proof}
It suffices to show that for each integer $n\geq0,$ all $x \in \Z_p,$ and $m\geq n,$
\begin{eqnarray*}
\phi_n(x+p^mz)) & \equiv & \phi_n(x) \pmod {p^{m-n}};\\
\psi_n(x+p^mz)) & \equiv & \psi_n(x) \pmod {p^{[(m-n)/p^k]}}.
\end{eqnarray*}
We leave the proof of these inequalities to the reader because it  follows from Definition \ref{ConZ}.
\end{proof}

The following result gives the Mahler expansion of $\phi_n.$

\begin{lemma}\label{skbinompre}
Let $\phi_n =\sum_{j=0}^{\infty}a_{j}^{(n)}\binom{x}{j}$ be the Mahler expansion of $\phi_n.$  Then the coefficients $a_{j}^{(n)}$ possess the following properties:

(1) $a_{j}^{(n)}=0$ for $0 \leq j <p^n;$

(2) $a_{j}^{(n)}=1$ for $ j =p^n;$

(3) If $j > p^n $, then $p$ divides  $a_{j}^{(n)}.$
\end{lemma}
\begin{proof}

We use Mahler's result to write $\phi_n =\sum_{j=0}^{\infty}a_{j}^{(n)}\binom{x}{j}.$ Then, from the well-known formula for
coefficients, we have
$$a_{j}^{(n)}=\sum_{i=0}^{j}(-1)^{j-i}\binom{j}{i}\phi_n(i).$$
We observe from the definition of $\phi_n$ that $\phi_n(i)=0$ for $0 \leq i <p^n$ and $\phi_n(p^n)=1.$
Then Parts (1) and (2) follow from these observations.
Furthermore, they also give
$$a_{j}^{(n)}=\sum_{i=p^n}^{j}(-1)^{j-i}\binom{j}{i}\phi_n(i).$$
For part (3), we use Lucas's congruence to show that $a_{j}^{(n)}\equiv 0 \pmod{p}$ for $j > p^n .$
Write $j$ and $i \geq p^n$ in $p$-adic form as
$$j=j_0 +j_1p+ \cdots +j_np^n + \cdots +j_sp^s;$$
$$i=i_0 +i_1p+ \cdots +i_np^n + \cdots +i_sp^s.$$
Note that $\phi_n(i) \equiv  i_n \pmod{p}$ if $i\geq p^n$ is of such $p$-adic form  with $i_n \not =0.$
Application of Lucas's congruence to $a_{j}^{(n)}$ gives
\begin{eqnarray}\label{lsum}
a_{j}^{(n)} \equiv \sum_{\substack{0\leq i_l \leq j_l, \forall l \not =n,\\ 1 \leq i_n \leq j_n}}(-1)^{j_0-i_0} \cdots (-1)^{j_s-i_s} \binom{j_0}{i_0} \cdots \binom{j_s}{i_s}i_n \pmod{p}.
\end{eqnarray}
If $j_l \not =0$ for some $l \not =n,$ then it is easy to see that
the sum in (\ref{lsum}) vanishes from the identity $$\sum_{0 \leq i_l \leq j_l} (-1)^{j_l-i_l}\binom{j_l}{i_l}=(1-1)^{j_l}.$$
If $j_l =0$ for all $l$ with $l\not= n,$ then $j=j_np^n >p^n$ with $j_n>1.$
Hence, the sum in (\ref{lsum}) vanishes as $$\sum_{1 \leq i_n \leq j_n} (-1)^{j_n-i_n}i_n\binom{j_n}{i_n}=j_n(-1)^{j_n-1}(1-1)^{j_n -1}.$$
We complete the proof.
\end{proof}

Parallel to Theorem \ref{sjrep}, we have the following theorem  for $\Zp.$
\begin{theorem}\label{sjrepinzp}
For any integer $j\geq0$, write $j=a_0+a_1p+ \cdots + a_{n}p^n$ with $0 \leq a_i <p.$
Set $$\Phi_j(x)=(\phi_0(x))^{a_0} (\phi_1(x))^{a_1} \cdots (\phi_n(x))^{a_n} ~(j>0),~ \Phi_0(x)=1.$$
Then, $\{\Phi_{j}\}_{j\geq 0}$ is an orthonormal basis for $C(\Zp,\Qp).$
\end{theorem}
\begin{proof}
We provide two proofs which rely on Conrad's digit principle \cite[Theorem 2]{Co2}.
First, we provide a direct proof by showing that for any integer $n>0$,
the map $\Zp /p^n\Zp \rightarrow (\Zp/p\Zp)^n$ defined by
$$x  \mapsto (\phi_0(x), \phi_1(x), \cdots,\phi_{n-1}(x)){\pmod p}$$
is a bijection.
Then, the map is well defined because of the observation that if $x\equiv y \pmod{p^n},$ then  $\phi_i(x)\equiv \phi_i(y)\pmod{p}$ for all $ 0 \leq i <n.$
Let us show that the map is bijective, which is equivalent to being surjective.
Writing $x =x_0+ x_1p+ \cdots + x_{n-1}p^{n-1} \in \Zp/p^n\Zp$ with $0 \leq x_i <p,$  the image of the map is then simply
$$(\phi_0(x), \phi_1(x), \cdots,\phi_{n-1}(x)){\pmod p} =(x_0, x_1, \cdots, x_{n-1}).$$
Then, the map is surjective, hence, bijective, completing the first proof.
We provide a second proof using Lemma  \ref{skbinompre} by which we have $$||\phi_n(x)-\binom{x}{p^n}|| \leq 1 /p <1.$$
These inequalities also imply
$$||\Phi_j(x)-\{ \begin{array}{cc}
  x \\j \end{array} \}|| \leq 1 /p <1,$$
where $$\{ \begin{array}{cc}
  x \\j \end{array} \} =\binom{x}{1}^{a_0} \binom{x}{p}^{a_1} \cdots\binom{x}{p^n}^{a_n} $$ for the $p$-adic representation of $j$ in Theorem \ref{sjrepinzp}.

Theorem \ref{sjrepinzp} now follows  by applying Lemma \ref{ba} to the inequality above,
together with \cite[Theorem 11]{Co2} which reads that $\{ \begin{array}{cc} x \\j \end{array} \}_{j \geq 0}$ is  an orthonormal basis for $C(\Zp,\Qp).$

\end{proof}
The following result is also parallel to Theorem \ref{sjrep}.
\begin{theorem}\label{sjrepinzp2}
For any integer $j\geq0$, write $j=a_0+a_1p+ \cdots + a_{n}p^n$ with $0 \leq a_i <p.$
Set $$\Psi_j(x)=(\phi_0(x))^{a_0} (\psi_1(x))^{a_1} \cdots (\psi_n(x))^{a_n} ~(j>0),~ \Psi_0(x)=1.$$
Then, $\{\Psi_{j}\}_{j\geq 0}$ is an orthonormal basis for $C(\Zp,\Qp).$
\end{theorem}
\begin{proof}
An independent proof follows by applying the same proof that was applied in Theorem \ref{sjrepinzp} to $\psi_n.$
An alternative proof follows from Theorem \ref{sjrepinzp} and Lemma \ref{ba} together with
the observation that for all $n\geq0,$ $$\phi_n \equiv \psi_n \pmod {p}$$
equivalently for all $j\geq0,$  $$\Phi_j \equiv \Psi_j \pmod {p}.$$
\end{proof}

Unlike the function field case, we were unable to determine a formula for the expansion coefficients in Theorems \ref{sjrepinzp} and \ref{sjrepinzp2}. Therefore, it may be interesting to find a closed-form formula for the coefficients in these two theorems.

\section{Cartier operators on more general settings}
In this section, we use the Cartier operators with more general settings than in the previous sections.
We show that the Hasse and Cartier operators defined previously satisfy  what is known as the  binomial inversion formula
in the sense that it resembles the well-known binomial inversion formula for two sequences of natural numbers.
Moreover, we employ Cartier operators  to present several Wronskian criteria for linear independence on the same settings.
\subsection{Binomial inversion formula for Hasse and Cartier operators.}
Let $\k$ be a perfect field of  positive characteristic $p$ and let $\k[[t]]$ be the ring of formal power series in one variable $t$ over $\k,$ equipped with the valuation $v$, such that the associated absolute value $|x|$ with $x \in \k[[t]]$ is defined by $|x| =e^{-v(x)},$ where $e $ is a real number $>1.$ Then, $|?|$ can be naturally extended to the quotient field of $\k[[t]]$ which we will denote by
$\k((t)).$ It is also well known that $\k((t))$ is a topological space under this absolute value.

For a comparison with the Cartier operators we first recall higher derivatives (also termed Hasse derivations).
The $k$-linear higher derivative $\{\cD_{n,t}\}_{n\geq 0}$ ($\cD_{n}=\cD_{n,t}$ in abbreviated notation)
is defined by
$$ \cD_{n}(t^m) =\binom{m}{n}t^{m-n}.$$

As continuous $\k$-linear operators, higher derivatives satisfy various properties such as the product formula and chain
rule and the reader can consult \cite{J5} for additional background information in this regard.
We hold $q=p^a$ constant  for a power of a prime $p,$ being characteristic of $\k,$ and
define two types of Cartier operators on $\k[[t]].$

\begin{definition}\label{threeC}
Let $n$ be an integer such that $ q^{k-1} \leq n < q^k$ for $k>0$ an integer  or $(n,k)=(0,0).$
Then, the Cartier operators $\{ \phi_n \}_{n\geq0}$ and $ \{ \psi_n\}_{n\geq0}$
are respectively defined by
\begin{eqnarray*}
\phi_n(\sum_{i\geq0}{x_i}t^i) &=& \sum_{i\geq0}x_{iq^k +n}t^{iq^k}; \\
\psi_n(\sum_{i\geq0}{x_i}t^i)&=&\sum_{i\geq0}x_{iq^k +n}^{1/q^k} t^{i}.
\end{eqnarray*}
\end{definition}

The assumption that $\k$ is a perfect field of characteristic $p>0$ is necessary for $\psi_n$ such that
coefficients of $\psi_n(x)$ lie in $\k.$
Note that $\{\phi_n\}_{n\geq0}$ is alternatively defined, for a monomial $t^m$, by
$$ \phi_n(t^m)=t^{m-r }\delta_{n,r},$$
for some $r,$ which is the remainder of the division of $m$ by $q^k$, where $k= [ loq_{q} n] +1,$;
that is, $m=lq^k+ r$ where $0 \leq r <q^k$ and $l\geq0.$

As in the proof of Lemma \ref{basic} (2), we deduce that for $ q^{k-1} \leq n < q^k,$ and $x, y \in \k[[t]].$
\begin{eqnarray*}
\phi_n(x^{q^k}y) &=& (x^{(k)})^{q^k}\phi_n(y); \\
\psi_n(x^{q^k}y) &=& x^{(k)}\psi_n(y), \nonumber
\end{eqnarray*}
where $x^{(k)} = \sum_{i\geq0}x_{i}^{q^k} t^{i}.$
The preceding identities imply that for $x\in \k[[t]],$
\begin{eqnarray}\label{qthpower1}
\phi_n(t^{q^k}x) &=& t^{q^k}\phi_n(x);\\
\psi_n(t^{q^k}x) &=& t\psi_n(x). \nonumber
\end{eqnarray}

\begin{lemma}
For each $n\geq0,$ $\phi_n$ and $\psi_n$ are a continuous $\k$-linear operator on $\k[[t]].$
\end{lemma}
\begin{proof}
For $\phi_n$ to be continuous, it suffices to show from the $\k$-linearity that
the following inequality holds:
\begin{eqnarray}\label{valphi}
v(\phi_n(x)) \geq v(x)-n,
\end{eqnarray}
implying the continuity of $f$ at $x=0.$
Writing $m:=v(x)=lq^k +s$ with $0 \leq s < q^k$ and $l\geq0,$ and
$x=t^my$ with $(t,y)=1,$ using (\ref{qthpower1}), we have
$$ v(\phi_n(x))\geq m-s + v(\phi_n(t^sy)).$$
A direct  estimation of $v(\phi_n(t^sy))$ gives the desired result in (\ref{valphi}).
As for $ \psi_n$, it follows from the inequality
$ v(\varphi_n(x))= v(\varphi_n(t^my)) \geq [m/q^k]$ as in Lemma \ref{continu}.
\end{proof}

\begin{lemma}\label{nega}
For any integer $m >0$ such that $m=lq^k +s $ with $0 \leq s <q^k$ and $l\geq0$ and for $n$  an integer such that
$1\leq n <q^k,$
\begin{equation*}
\sum_{r=n}^{q^k-1}\binom{r}{n} \binom{m+r-1}{r} \equiv  \left \{ \begin{array}{ll}
(-1)^{n} \pmod{p}    & \mbox{if}~~ n+s=q^k,\\
 0 \pmod{p} & \mbox{otherwise}.~~
 \end{array}
 \right.
\end{equation*}
\end{lemma}

\begin{proof}
By the well-known identities for binomial coefficients, we see
$$\sum_{r=n}^{q^k-1}\binom{r}{n} \binom{m+r-1}{r}  =\binom{m+n-1}{m-1}\sum_{i=0}^{q^k-1-n}\binom{m+n+i-1}{m+n-1}.$$
Because the sum above equals the coefficient of $x^{m+n}$ in the polynomial of the form
$(1+x)^{m+n-1} + \cdots + (1+x)^{m+q^k-1},$ we obtain

\begin{eqnarray*}
\binom{m+n-1}{m-1}\sum_{i=0}^{q^k -1-n} \binom{m+n+i-1}{m+n-1} & = & \binom{m+n-1}{m-1}\binom{m+q^k -1}{m+n} \\
 &=& \binom{lq^k +s+n-1} {lq^k +s-1} \binom{lq^k +q^k +s-1}{lq^k +s+n}.
\end{eqnarray*}

If $n+s=q^k$ then by Lucas's congruence, we have
$$\binom{lq^k +s+n-1} {lq^k +s-1} \binom{lq^k +q^k +s-1}{lq^k +s+n}
\equiv \binom{q^k-1}{s-1} \equiv (-1)^{s-1} \equiv (-1)^{n} \pmod{p}.$$

In general, writing $n +s =\e q^k +j \not=q^k$ for some $\e \in \{0,1\}$ and $0\leq j <q^k-1,$
we note that if $\e=0$ then $j\geq1,$ and that if $\e=1$ then $j=0$ is excluded, that is $j\geq1.$
For these cases, Lucas's congruence
gives
$$ \binom{lq^k +s+n-1} {lq^k +s-1} \binom{lq^k +q^k +s-1}{lq^k +s+n} \equiv
\binom{l+\e}{l}\binom{j-1}{s-1}\binom{l+1}{l+\e}\binom{s-1}{j}\equiv 0 \pmod {p}.$$
Then we have the result, as desired.
\end{proof}

By extending Lemma \ref{bif} or Theorems \ref{singleout} and \ref{inv} to a complete generality we  have the
binomial inversion formula for the Hasse and Cartier operators.
\begin{theorem}\label{gmain}
Let $ q^{k-1} \leq n < q^k$ or $(n,k)=(0,0)$ and let $x \in \k((t)).$
Then,
\begin{itemize}
\item[] {\rm (1)} $\phi_{n}(x) \displaystyle = \sum_{r=n}^{q^k -1}\binom{r}{n}(-t)^{r-n}\cD_{r}(x).$
\item[] {\rm (2)} $\cD_n(x) \displaystyle = \sum_{r=n}^{q^k -1}\binom{r}{n}t^{r-n}\phi_{r}(x).$
\item[] {\rm (3)} $\phi_{q^k-1}(x) =\cD_{q^k -1}(x).$
\end{itemize}

\end{theorem}

\begin{proof}
For part (1), it suffices to show that the two functions on both sides are identical
for $x=t^m (m\geq0)$ by means of both continuity and linearity.

 Case 1, in which $x \in \k[[t]].$  Writing $m=lq^k+ s$ where $0 \leq s<q^k$ and $l\geq0,$
 we have
 \begin{eqnarray*}
\sum_{r=n}^{q^k -1}\binom{r}{n}(-t)^{r-n}\cD_{r}(t^m) &=& \sum_{r=n}^{q^k -1}\binom{r}{n}(-t)^{r-n}\binom{m}{r}t^{m-r}
=t^{m-n}\sum_{r=n}^{q^k -1}(-1)^{r-n}\binom{r}{n}\binom{s}{r} \\
&=& t^{m-n}\sum_{r=n}^{q^k -1}(-1)^{r-n}\binom{s-n}{r-n}\binom{s}{n}=\binom{s}{n}t^{m-n} (1-1)^{s-n}\\
&=&t^{m-n}\delta_{n,s}=\phi_n(t^m).
\end{eqnarray*}

 Case 2, in which $x \in \k((t))).$ Suppose that $x= \a/\b$ for some $\a, \b \in \k[[t]].$
 Then, we may assume that  $\a/\b \not \in \k[[t]],$; thus, there exists the smallest  integer $l>0$   for which
 $t^l \a/\b \in  \k[[t]].$ Now, it suffices to verify that identity (1) holds for $t^{-m}$ where $0 <m \leq l.$
 We define $ \phi_n(t^{-m})$ on $\k((t))$ so that for $m=lq^k +s$ with $0\leq s<q^k,$
\begin{equation}\label{exten1}
\phi_{n}(t^{-m}) = \left \{ \begin{array}{ll}
    t^{-(m+n)} & \mbox{if}~~ n+s=q^k;
\\ 0   & \mbox{otherwise}.
 \end{array}
 \right.
\end{equation}
Alternatively, $\phi_n$ can be extended  to $\k((t))$
by setting $\phi_n(t^{-m})=\phi_n(t^{m(q^k-1)})/t^{mq^k}.$
Now, we use Lemma \ref{nega} to compute:
\begin{eqnarray*}
\sum_{r=n}^{q^k -1}\binom{r}{n}(-t)^{r-n}\cD_{r}(t^{-m}) &=& \sum_{r=n}^{q^k -1}(-1)^{r-n}\binom{r}{n} \binom{-m}{r}t^{-m-n} \\
&=& \sum_{r=n}^{q^k -1}(-1)^{n}\binom{r}{n} \binom{m+r-1}{r}t^{-m-n} \\
&=& (-1)^{n+s-1}t^{-m-n}\delta_{n+s,q^k}\\
&=&\phi_n(t^{-m}),
\end{eqnarray*}
where the last equality follows from (\ref{exten1}).

It is not difficult to establish that part (2) follows from part (1) by substituting (1) into (2) as in Theorem \ref{inv}.
Alternatively, the proof could be obtained by applying the same argument as in part (1).
Moreover, part (3) follows from parts (1) or (2).
\end{proof}

The following  identity is parallel to (\ref{id3R}): For $q^{k-1} \leq n < q^{k},$
\begin{eqnarray}\label{trans}
\phi_{n}(x) = \phi_{q^k-1}(t^{q^k-1-n}x),
\end{eqnarray}
which follows from the definitions of $\phi_n$ on $\k[[t]]$ and its extension to $\k((t))$ in (\ref{exten1}).

\begin{remark}{\rm

Formula (1) in Theorem \ref{gmain} provides that  $\phi_n$ can be a suitable alternative for calculating the higher derivatives of any functions in $\k[[t]]$; therefore, it can play a role in the study of rigid analytic functions occurring in a field of positive characteristic as a substitute for higher derivatives. For example, it was shown in \cite {BP} and \cite{AP} that higher derivatives are extensively used to  investigate the differential properties of Drinfeld quasi-modular forms and to derive the arithmetic properties of the maximal extension of $\F(T)$ which is abelian and tamely ramified at the infinity prime.
It would be interesting to establish the results parallel to these results by replacing the higher derivatives replaced with Cartier operators. As an illustration, the next section is devoted to some Wronskian criteria associated with the Cartier operators on more general settings.
}
\end{remark}
We use the binomial inversion formula in Theorem \ref{gmain} to derive a product formula for $\phi_n$ in the case where $n<q.$
\begin{theorem}\label{prodfor}
For $ 1\leq n<q$ and $x, y \in \k((t)),$
$$\phi_{n}(xy) \displaystyle = \sum_{i+j=n}\phi_{i}(x)\phi_{j}(y) +t^{q}\sum_{i+j=q +n}\phi_{i}(x)\phi_{j}(y).$$
In particular,
$$\phi_{q-1}(xy) \displaystyle = \sum_{i+j=q-1}\phi_{i}(x)\phi_{j}(y).$$
\end{theorem}

\begin{proof}
For a positive integer $n < q,$  we use the formula (1) in Theorem \ref{gmain} to have
$$\phi_{n}(xy) \displaystyle = \sum_{r=n}^{q -1}\binom{r}{n}(-t)^{r-n}\cD_{r}(xy).$$
The product formula of $\cD_r$ enables us to rewrite it as
$$\phi_{n}(xy) \displaystyle  =\sum_{r=n}^{q -1}\binom{r}{n}(-t)^{r-n}\sum_{\a +\b=r} \cD_{\a}(x)\cD_{\b}(y).$$
By formula (2) in Theorem \ref{gmain}
it equals
\begin{eqnarray*}
 \phi_{n}(xy) \displaystyle  &=& \sum_{r=n}^{q -1}\binom{r}{n}(-t)^{r-n}\sum_{\a+ \b=r}
\left(\sum_{i=\a}^{q-1}\binom{i}{\a}t^{i-\a}\phi_i(x)
\sum_{j=\b}^{q-1}\binom{j}{\b}t^{j-\b}\phi_j(y)\right)\\
&=&\sum_{r=n}^{q -1}\binom{r}{n}(-t)^{r-n}\sum_{\a+ \b=r}
\left(\sum_{i=\a}^{q-1}\sum_{j=\b}^{q-1} \binom{i}{\a}\binom{j}{\b}
t^{i+j-r}\phi_i(x)\phi_j(y)\right)\\
&=& \sum_{r=n}^{q -1}(-1)^{r-n}
\sum_{i=\a}^{q-1}\sum_{j=\b}^{q-1} \left(\sum_{\a+ \b=r}\binom{i}{\a}\binom{j}{\b}\right)
t^{i+j-n}\phi_i(x)\phi_j(y).
\end{eqnarray*}
From well-known formulas for binomial coefficients,
\begin{eqnarray*}
\phi_{n}(xy) \displaystyle  &=&
\sum_{i=\a}^{q-1}\sum_{j=\b}^{q-1} \sum_{r=n}^{q -1}(-1)^{r-n}\binom{r}{n} \binom{i+j}{r}t^{i+j-n}\phi_i(x)\phi_j(y)\\
&=& \sum_{i=\a}^{q-1}\sum_{j=\b}^{q-1} \left( \binom{i+j}{n}\sum_{r=n}^{q-1} (-1)^{r-n}\binom{i+j}{r-n} \right) t^{i+j-n}\phi_i(x)\phi_j(y).
\end{eqnarray*}
The identity $\binom{i+j}{n}\sum_{r=n}^{q-1}(-1)^{r-n}\binom{i+j-n}{r-n} = \delta_{i+j,n}$ or $\delta_{i+j, q +n}$ leads to
the desired result.

Part (2) follows immediately from part (1).
\end{proof}

In general, it is of interest to remove the restriction on $n$ in the product formula in Theorem \ref{prodfor}.

\subsection{Wronskian criteria associated with Cartier operators }
We retain all notation from the previous section except for $q$ being a prime $p.$
As stated in the previous section, we present several Wronskian criteria for the linear independence
of a finite number of element in $\k((t))$ over certain of its subfields.
To do this, let us define all the necessary notations.

As in \cite{GV}, setting $K=\k((t)),$ let
\begin{eqnarray*}
K_m^{\cD} &=&  \{ x \in K ~|~ \cD_i(x)=0~ for  ~1 \leq i <p^m \} \\ \label{Hid4}
K_{\infty}^{\cD} &=& \{ x \in K ~|~ \cD_i(x)=0~ for ~ 1 \leq i  \}. \label{Hid5}
\end{eqnarray*}
Parallel to these, set
\begin{eqnarray*}
K_m^{\phi} &=&  \{ x \in K ~|~ \phi_i(x)=0~ for  ~1 \leq i <p^m \} \\ \label{Cid4}
K_{\infty}^{\phi} &=& \{ x \in K ~|~ \phi_i(x)=0~ for ~ 1 \leq i  \}. \label{Cid5}
\end{eqnarray*}

From Theorem \ref{gmain} we note that $K_m^{\phi}=K_m^{\cD}$ and $K_{\infty}^{\phi}=K_{\infty}^{\cD},$; thus,
hereafter we denote by  $K_m =K_m^{\cD} =K_m^{\phi}$ and by $ K_{\infty} =K_{\infty}^{\cD} = K_{\infty}^{\phi}.$

We have the following result for Cartier operators which is analogous to that of Schmidt \cite[Satz 10]{HS}
for Hasse derivatives.
\begin{theorem}\label{finiteE}
Let $\k$ be a perfect field of positive characteristic $p$ and let $F$ be any field with $\k(t) \subset F \subset\k((t)) ,$ a finite algebraic extension over $\k(t).$
Then, $\{ \phi_n \}_{n\geq0}$ and $\{ \psi_n \}_{n\geq0}$ extend uniquely to $F$ such that $F_m=\k F^{p^m}$ and $F_{\infty}^{\phi}=\k.$
\end{theorem}

\begin{proof}
First, we prove that $\{ \psi_n\}_{n\geq0}$  extends uniquely to $F,$; hence, this also occurs for $\{ \phi_n\}_{n\geq0}.$
Because $F$ is separable, there exists $y \in F$ such that $F=\k(t)(y)$ with $f(T,y)=0$ and $\frac{\partial f}{\partial y} \not=0.$
Then, it is well documented in \cite{Sc}  that $\{\cD_n \}_{n\geq0}$ extends uniquely to $F$ with $\cD_n(F) \subset F$ for all $n\geq0.$
For $p^{k-1} \leq n < p^{k},$ using (\ref{trans}) and (3) in Theorem \ref{gmain} we have
$$ \phi_{n}(y)=\phi_{p^k-1}(T^{p^k-1-n}y)= \cD_{p^k -1}(T^{p^k-1-n}y) \in F.$$
Equivalently,
$$ \psi_{n}^{p^k}(y)=\cD_{p^k -1}(T^{p^k-1-n}y).$$
As $F$ is a separable extension,
$\cD_{p^k -1}(T^{p^k-1-n}y)$ is a $p^k$ th power in $F,$ such that $ \psi_{n}(y)$ belongs to $F.$
For additional properties, as $K_m^{\phi}=K_m^{\cD}$ and $K_{\infty}^{\phi}=K_{\infty}^{\cD},$
it is straightforward to verify that $F_m^{\phi}=\k F^{p^m}$ and $F_{\infty}^{\phi}=\k.$
\end{proof}

For any vector $\e:=(\e_0, \cdots,\e_n)$ of integers $\e_{i}$ with $ 0 \leq \e_{0} < \cdots <
\e_{n},$ we define two Wronskians of a family of Laurent series $x_0, \cdots,x_{n} \in K$ as
$$W_{\e}^{\phi}(x_0, \cdots,x_{n}):={\rm det}( \phi_{\e_i}(x_j))_{0 \leq i,j \leq n}
$$ and
$$W_{\e}^{\psi}(x_0, \cdots,x_{n}):={\rm det}( \psi_{\e_i}(x_j))_{0 \leq i,j \leq n}.
$$

In parallel with  the Wronskian criterion associated with $\cD_{n}$, according to Schmidt \cite{Sc},
we provide Wronskian criteria associated with  $\phi_n$  and $\psi_n.$
\begin{theorem}\label{Wcar}
Let $x_0, \cdots, x_{n}$ be elements in $K.$ Then the following are equivalent:

(1)  $x_0,\cdots, x_{n}$ are linearly independent over $\k.$

(2) There exists a  sequence of integers $\e_{i}$ with $0\leq \e_{0} < \cdots < \e_{n}$  such that
$W_{\e}^{\phi}(x_0, \cdots,x_{n}) \not =0.$

(3) There exists a  sequence of integers $\e_{i}$ with $0\leq \e_{0} < \cdots < \e_{n}$ such that
$W_{\e}^{\psi}(x_0, \cdots,x_{n}) \not =0.$

\end{theorem}

\begin{proof}
We provide two proofs of equivalence (1)$\Leftrightarrow$ (2); the first proof using the result of Schmidt \cite{Sc}, and
a second independent proof.
By using Theorem \ref{gmain} we are able to determine that the former proof proceeds in the same way as that of Theorem \ref{Wcar2}.
The proof of equivalence (1)$\Leftrightarrow$ (3) follows in the same way as the second proof of (1)$\Leftrightarrow$ (2)
because there is no known formula for $\psi_n$ such as Theorem \ref{gmain}.
\end{proof}

Now, we provide the second proof of equivalence (1)$\Leftrightarrow$ (2), independent to the first,
for which two lemmas concerning power series are needed.
To state the lemma we recall that the order or the $t$-adic valuation of a nonzero power series is the
smallest exponent with a nonzero coefficient in that series.
\begin{lemma}\label{WL1}
Let $\k$ be a field of arbitrary  characteristic and let $f_0, \cdots f_{n}$ be a family of power series in $\k[[t]]$ which are linearly independent over $\k.$
Then, there exists an invertible $(n+1)\times (n+1)$ matrix $A$ with entries in $\k$ such that the power series $
g_0, \cdots g_n$ defined by
$$ [g_0, \cdots, g_n ] = [f_0, \cdots, f_n]A$$ are all nonzero and have mutually disjoint orders.
Consequently, for any vector $\e=(\e_0, \cdots, \e_n)$ of integers $\e_{i}$  with $0\leq\e_{0} < \cdots < \e_{n},$
 the following equalities hold:
$$ W_{\e}^{\phi}(g_0, \cdots,g_n)= W_{\e}^{\phi}(f_0, \cdots,f_n){\rm det}(A)$$
and
$$ W_{\e}^{\psi}(g_0, \cdots,g_n)= W_{\e}^{\psi}(f_0, \cdots,f_n){\rm det}(A).$$
\end{lemma}
\begin{proof}
See \cite[Lemma 2]{BD}.
\end{proof}

\begin{lemma}\label{WL2}
Let $\k$ be a field of arbitrary characteristic. If the nonzero power series $g_0, \cdots g_n$ in $\k[[T]]$ have orders $\e=(\e_0, \cdots,\e_n)$ such that the orders $\e_i$ of $g_i$'s are arranged in increasing order, then
$W_{\e}^{\phi}(g_0, \cdots,g_n) \not =0$ and $W_{\e}^{\psi}(g_0, \cdots,g_n) \not =0.$
\end{lemma}
\begin{proof}
It follows by mimicking the proof of \cite[Lemma 3]{BD}.
\end{proof}

\noindent{\bf Second proof of (1) $\Leftrightarrow$ (2) in Theorem \ref{Wcar}.}

\begin{proof}
We merely prove the "only if" part because the "if" part is obvious.

Case 1, in which $x_i$'s are in $R=\k[[T]].$
Given $\k$-linearly independent vectors $x_i( 0\leq i \leq n),$ by Lemmas \ref{WL1} and \ref{WL2}, we can accept $\e_i$ as the order of $g_i$ such that
\begin{eqnarray*}
 0 &\leq& \e_0 < \e_1 \cdots <\e_n,\\
g_i &=&\sum_{j=0}^{n}a_{ji}x_j~~ {\rm with}~~ {\rm det}(a_{ji})\not =0.\\
\end{eqnarray*}
As $W_{\e}^{\phi}(g_0, \cdots,g_n) \not =0$,  we deduce $W_{\e}^{\phi}(x_0, \cdots,x_n) \not =0$ from the above equation.

Case 2, in which the $x_i$'s  are in $K.$
Then, there exists a positive integer $r$ such that $t^r x_i( 0\leq i \leq n)$ are in $R.$
Note that $t^r x_i \in R $ are linearly independent over $\k$ if and only if
this is also true for $x_i \in K.$ By Case 1,
there exists a vector $\e=(\e_0, \cdots,\e_n )$ with $ 0 \leq \e_0 < \e_1 \cdots <\e_n,$ and an invertible  matrix $A$ such that
 $$W_{\e}^{\phi}(g_0, \cdots,g_n)= W_{\e}^{\phi}(t^rx_0, \cdots, t^rx_n)det(A).$$
If we choose a $p$th power $r=p^l$  such that $r >e_n$, then  we have
$$W_{\e}^{\phi}(t^rx_0, \cdots, t^rx_n) = t^{r(n+1)}W_{\e}^{\phi}(x_0, \cdots ,x_n).$$
From this equality we deduce $W_{\e}^{\phi}(x_1, \cdots, x_n)\not =0$ as $W_{\e}^{\phi}(g_0, \cdots,g_n) \not =0.$
Thus, the proof of the "only if" part is complete.
\end{proof}

The following result for Cartier operators $\phi_n$ is analogous to the Wronskian criterion associated with
the higher derivatives $\cD_n$, based on the results of Garcia and Voloch in \cite[Theorem 1]{GV}.

\begin{theorem}\label{Wcar2}
Let $x_0, \cdots, x_{n}( n<p^m)$ be elements in $K.$ Then,  $x_0,
\cdots, x_{n}$ are linearly independent over $K_m$  if and only if
there exists a sequence of integers $\e_{i}$ with $
0\leq \e_{0} < \cdots < \e_{n}<p^m$ such that
$ W_{\e}^{\phi}(x_0, \cdots,x_{n}) \not =0.$
\end{theorem}
\begin{proof}
Assuming that $x_0, \cdots, x_{n}$ are linearly dependent over $K_m,$
then there exist $\a_0, \cdots\a_n \in K_m$ not all of which are zero such that
$\sum_{j=0}^{n}\a_jx_j=0.$ Because $K_m=\k K^{p^m}$
we have $\sum_{j=0}^{n}\a_j\phi_{\e_i}(x_j)=$ for all $0\leq e_0, \cdots,\e_n <p^m.$
This proves the "if" part of the theorem.

Conversely, if  $x_0,\cdots, x_{n}$ are linearly independent over $K_m$ then, by \cite[Theorem 1]{GV},
there exists a sequence of integers $\e_{i}$ with $
0\leq \e_{0} < \cdots < \e_{n}<q^m$ such that
$$ W_{\e}^{\cD}(x_0, \cdots,x_{n}):= {\rm det}( \cD_{\e_i}(x_j))_{0 \leq i,j \leq n} \not =0,$$
where $W_{\e}^{\cD}$ represents the Wronskian associated with the higher derivatives.
From identity (2) in Theorem \ref{gmain}, the $(n+1)\times (n+1)$ invertible
matrix $ D:=(\cD_{\e_i}(x_j))_{0 \leq i,j \leq n}$ has the following decomposition:
\begin{eqnarray}\label{Mat}
D =M \Phi,
\end{eqnarray}
where $M=(M_{\e_i,j})$ is a  $(n+1) \times p^m$ matrix whose $(i,j)$ entry is given by formula (2) in Theorem \ref{gmain}
and the $p^m \times (n+1)$ matrix $\Phi =( \phi_{i}(x_j))_{0 \leq i\leq p^m -1, 0\leq j \leq n}.$
Computing the rank of  the transpose of the matrix equality in (\ref{Mat})
yields  $${\rm rank}(D^T) \leq {\rm rank}(\Phi^T),$$
implying $\Phi$ has full rank; equivalently, there exist integers
$\e_{i}$ with $0\leq \e_{0} < \cdots < \e_{n}<p^m$
for which the corresponding row vectors of $\Phi$ are linearly independent over $K.$

\end{proof}
Several remarks are in order.

1. It is of interest to provide the proof of Theorem \ref{Wcar2}, independent of the result of
Garcia and Voloch \cite[Theorem 1]{GV}.

2. It is much simpler and more practical to compute $W_{\e}^{\phi}(x_0, \cdots,x_{n})$ than $ W_{\e}^{\cD}(x_0, \cdots,x_{n}),$
because the calculation of $W_{\e}^{\phi}(x_0, \cdots,x_{n})$ does not concern binomial coefficients.
It is interesting to check that $W_{\e}^{\phi}(x_0, \cdots,x_{n}) \not =0 $ if and only if
$ W_{\e}^{\cD}(x_0, \cdots,x_{n})\not =0 $ for the same integer vector $\e=(\e_0, \cdots,\e_n).$

3. The question arises as to whether there exists a Wronskian criterion associated with Cartier
operators $\psi_n.$


\indent



\end{document}